\documentclass[twocolumn,10pt]{asme2e}

\usepackage{amssymb}
\usepackage{amsmath}         
\usepackage{mathrsfs}        
\usepackage{enumerate}       
\usepackage{graphicx}
\usepackage[numbers,sort&compress]{natbib}
\usepackage[colorlinks,
            bookmarksopen,
            bookmarksnumbered,
            linkcolor=blue,
            anchorcolor=blue,
            citecolor=blue]{hyperref}
\newtheorem{theorem}{Theorem}
\newtheorem{definition}{Definition}
\newtheorem{lemma}{Lemma}
\newtheorem{remark}{Remark}                             \newtheorem{corollary}{Corollary}
\newtheorem{example}{Example}
\title{Some fundamental properties on the sampling free\\ nabla Laplace transform}

\author{Yiheng Wei, Yuquan Chen, Yong Wang
    \affiliation{
	Lab of Vibration Control and Vehicle Control\\
	Department of Automation\\
	University of Science and Technology of China\\
	Hefei, Anhui 230026, China\\
    }	
}
\author{YangQuan Chen*
    \affiliation{
	Lab of Mechatronics, Embedded Systems and Automation\\
	School of Engineering\\
    University of California, Merced\\
    5200 North Lake Road, Merced, CA 95343, USA\\
    Email: ychen53@ucmerced.edu
    }	
}

\begin{document}

\maketitle

\begin{abstract}
{\it Discrete fractional order systems have attracted more and more attention in recent years. Nabla Laplace transform is an important tool to deal with the problem of nabla discrete fractional order systems, but there is still much room for its development. In this paper, 14 lemmas are listed to conclude the existing properties and 14 theorems are developed to describe the innovative features. On one hand, these properties make the $N$-transform more effective and efficient. On the other hand, they enrich the discrete fractional order system theory.}
\end{abstract}

%

\section{INTRODUCTION}
Fractional calculus has gained importance during the past three decades due to its applicability in diverse fields of science and engineering \cite{Monje:2010Book,Ma:2016JCND,Ma:2018JCND,Sun:2018CNSNS,Ma:2019F}. Especially, with the rapid development of digital computers, more and more attention has been attracted on the discrete fractional calculus and there has been a continuing growth in the number of research outcomes on this aspects, such as, domain transform \cite{Ortigueira:2015SP,Ortigueira:2016CNSNS}, short memory principle \cite{Wei:2017FCAA,Wei:2019CNSNSa}, adaptive filtering \cite{Cheng:2017ISA,Liu:2019SP}, system identification \cite{Cui:2018ISA,Cheng:2018SP}, and stability analysis \cite{Wei:2018ISA,Baleanu:2017CNSNS}, etc.

The analogous theory for nabla discrete fractional calculus was initiated and properties of the theory of fractional sums and differences were established \cite{Diaz:1974MC,Gray:1988MC}. Recently, a powerful technique, the sampling free nabla Laplace transform, was developed to deal with the discrete fractional calculus in the sense of a backward difference \cite{Atici:2009EJQTDE}. The existence and uniqueness of nabla Laplace transfrom were discussed in \cite{Goodrich:2015Book}. After deriving the linearity and shifting properties, such transform was applied to solve some nabla Caputo fractional difference equations with initial value problems \cite{Mohan:2014CMS}. The inverse nabla Laplace transform was formulated and then with the help of the nabla Laplace transform, an effective numerical implementation method was proposed for nabla discrete fractional order systems \cite{Wei:2019AJC}. By expressing a nabla discrete fractional order system in nabla domain, its infinite dimensional characteristics was demonstrated plainly \cite{Wei:2019CNSNSb}. After exploring and applying some typical properties of this transform, the montonicity and overshoot property of the time-domain response was studied for linear-time-invariant nabla discrete fractional order system \cite{Wei:2019CNSNSc}.

Despite the plentiful achievements, there still exists some room for further investigation. The existing references almost focus on its applications. Compared with the well adopted Laplace transform and $Z$-transform \cite{Lwiki,Zwiki}, the nature properties exploration for nabla Laplace transform itself is far from adequate. The nabla Laplace transform with fully known properties will play a critical role in the analysis and synthesis of nabla discrete fractional order systems. However, how to achieve this step is still an open problem. All of these motivate this study. 

The remainder of this paper is organized as follows. Section \ref{Section 2} introduces some preliminary definitions. Some fundamental properties are investigated systematically in Section \ref{Section 3} where 14 novel properties are formulated by theorems. In Section \ref{Section 4}, several numerical applications are provided by using the considered tool. Finally, Section \ref{Section 5} draws some conclusions.
\section{PRELIMINARIES}\label{Section 2}
In this section, the basic knowledge for the nabla calculus and the $N$-transform will be provided.

\subsection{Nabla Calculus}
\begin{definition} \label{Definition 1}
(\textbf{Nabla difference} \cite{Goodrich:2015Book})
For $x:\mathbb{N}_{a+1-n} \to \mathbb{R}$, its $n$-th nabla difference is defined by
\begin{equation}\label{Eq1}
{\textstyle {\nabla ^n}x\left( k \right) \triangleq \sum\nolimits_{j = 0}^{n} {{{{\left(\hspace{-1pt}{ - 1} \hspace{-1pt}\right)}^j}\left({\begin{smallmatrix}
{ n }\\
j
\end{smallmatrix}}\right)}x\left( {k - j} \right)},}
\end{equation}
where $n\in\mathbb{Z}_+$, $k\in \mathbb{N}_{a+1}$, $a\in\mathbb{R}$, $\mathbb{N}_{a+1}\triangleq\{a+1,a+2,a+3,\cdots\}$, ${\left( {\begin{smallmatrix}
{ p }\\
q
\end{smallmatrix}} \right)} \triangleq \frac{{\Gamma \left( {p+1  } \right)}}{{\Gamma \left( {q+1} \right)\Gamma \left( {p-q+1} \right)}}$ and $\Gamma\left(\cdot\right)$ is the Gamma function.
\end{definition}

\begin{definition} \label{Definition 2}
(\textbf{Nabla sum} \cite{Goodrich:2015Book})
For $x:\mathbb{N}_{a+1} \to \mathbb{R}$, its $n$-th nabla sum is defined by
\begin{equation}\label{Eq2}
{\textstyle { {}_a^{}\nabla_k^{-n} }x\left( k \right)  \triangleq \sum\nolimits_{j = 0}^{k-a-1} {{{\left( { - 1} \right)}^{j}}\big( {\begin{smallmatrix}
{ - n }\\
j
\end{smallmatrix}} \big)x\left( {k - j}\right)}},
\end{equation}
where $n\in\mathbb{Z}_+$, $k\in \mathbb{N}_{a+1}$ and $a\in\mathbb{R}$.
\end{definition}

To avoid the singularity, i.e., $\Gamma \left( { - n} \right) = \infty $, $n \in \mathbb{N}$, an alternative expression $\big( {\begin{smallmatrix}
{ - n}\\
j
\end{smallmatrix}} \big) = {\left( { - 1} \right)^j}\frac{{\Gamma \left( {n + j} \right)}}{{\Gamma \left( {j + 1} \right)\Gamma \left( n \right)}}$ can be obtained with the help of $\Gamma \left( z\right)\Gamma \left( {1 - z} \right) = \frac{\pi }{{\sin \left( {z\pi } \right)}}$.

It can be found that when $n=0$ is given for Definition \ref{Definition 1}, a reasonable result follows immediately
\begin{equation}\label{Eq3}
{\textstyle {\nabla ^0}x\left( k \right) = x\left( {k } \right).}
\end{equation}
Similarly, when the same condition is exerted for Definition \ref{Definition 2}, it leads to a similar formula
\begin{equation}\label{Eq4}
{\textstyle { {}_a^{}\nabla_k^{0} }x\left( k \right) = x\left( {k } \right).}
\end{equation}

Extending the range of $n$ in Definition \ref{Definition 2}, then a generalized sum can be obtained.

\begin{definition} \label{Definition 3}
(\textbf{Nabla fractional sum} \cite{Goodrich:2015Book})
For $x:\mathbb{N}_{a+1} \to \mathbb{R}$, its $\alpha$-th nabla fractional sum is defined by
\begin{equation}\label{Eq5}
{\textstyle { {}_a^{}\nabla_k^{-\alpha} }x\left( k \right)  \triangleq \sum\nolimits_{j = 0}^{k-a-1} {{{\left( { - 1} \right)}^{j}}\big( {\begin{smallmatrix}
{ - \alpha }\\
j
\end{smallmatrix}} \big)x\left( {k - j}\right)}},
\end{equation}
where $\alpha\in\mathbb{R}_+$, $k\in \mathbb{N}_{a+1}$ and $a\in\mathbb{R}$.
\end{definition}

With the well-established calculus, the nabla fractional difference can be presented now.

\begin{definition} \label{Definition 4}
(\textbf{Nabla Caputo fractional difference} \cite{Goodrich:2015Book})
For $x:\mathbb{N}_{a+1-n} \to \mathbb{R}$, its $\alpha$-th nabla Caputo fractional difference is defined by
\begin{equation}\label{Eq6}
{\textstyle {{}_a^{\rm C}\nabla_k ^\alpha }x\left( k \right) \triangleq {{}_a^{}\nabla_k ^{ \alpha-n }}{\nabla ^n}x\left( k \right),}
\end{equation}
where $\alpha\in(n-1,n)$, $n\in\mathbb{Z}_+$, $k\in \mathbb{N}_{a+1}$ and $a\in\mathbb{R}$.
\end{definition}

\begin{definition} \label{Definition 5}
(\textbf{Nabla Riemann--Liouville fractional difference} \cite{Goodrich:2015Book})
For $x:\mathbb{N}_{a+1-n} \to \mathbb{R}$, its $\alpha$-th nabla Riemann--Liouville fractional difference is defined by
\begin{equation}\label{Eq7}
{\textstyle {{}_a^{\rm R}\nabla_k ^\alpha }x\left( k \right) \triangleq {\nabla ^n}{{}_a^{}\nabla_k ^{ \alpha-n }}x\left( k \right),}
\end{equation}
where $\alpha\in(n-1,n)$, $n\in\mathbb{Z}_+$, $k\in \mathbb{N}_{a+1}$ and $a\in\mathbb{R}$.
\end{definition}

\begin{definition} \label{Definition 6}
(\textbf{Nabla Gr\"{u}nwald--Letnikov fractional difference} \cite{Ostalczyk:2015Book})
For $x:\mathbb{N}_{a+1} \to \mathbb{R}$, its $\alpha$-th nabla Gr\"{u}nwald--Letnikov fractional difference is defined by
\begin{equation}\label{Eq8}
{\textstyle {{}_a^{\rm G}\nabla_k ^\alpha }x\left( k \right) \triangleq \sum\nolimits_{j = 0}^{k-a-1} {{{\left( { - 1} \right)}^{j}}\big( {\begin{smallmatrix}
{ \alpha }\\
j
\end{smallmatrix}} \big)x\left( {k - j}\right)}},
\end{equation}
where $\alpha\in(n-1,n)$, $n\in\mathbb{Z}_+$, $k\in \mathbb{N}_{a+1}$ and $a\in\mathbb{R}$.
\end{definition}

Because of the memory problem, ${{}_a^{\rm G}\nabla_k ^n }x\left( k \right)$ is not identical to $\nabla ^n x\left( k \right)$ in Definition \ref{Definition 1}, even though $\alpha=n\in\mathbb{Z}_+$ is adopted in Definition \ref{Definition 6}.

\subsection{Domain Transform}
\begin{definition} \label{Definition 7}
(\textbf{Laplace transform} \cite{Lwiki})
The Laplace transform of a function $x(t)$, defined for all real numbers $t \ge 0$, is the function $X(s)$, which is a unilateral transform defined by
\begin{equation}\label{Eq9}
{\textstyle {{\mathcal L}}\left\{ {x\left( t \right)} \right\} \triangleq \int_0^{ + \infty } {{{\rm{e}}^{ - st}}x\left( t \right){\rm{d}}t} ,}
\end{equation}
where $s$ is a complex number frequency parameter.
\end{definition}

As well known, the Laplace transform is used for the continuous time signal. For the discrete time case, the following $Z$-transform performs well.

\begin{definition} \label{Definition 8}
(\textbf{Bilateral $Z$-transform} \cite{Zwiki})
The bilateral $Z$-transform of a discrete time signal $x(k)$ is the formal power series $X(z)$ defined as
\begin{equation}\label{Eq10}
{\textstyle{} {\mathcal Z} \left\{ {x\left( {k} \right)} \right\} \triangleq \sum\nolimits_{k = -\infty}^{ + \infty } {z^{-k}x\left( {k} \right) } ,}
\end{equation}
where $k$ is an integer and $z$ is, in general, a complex number.
\end{definition}

\begin{definition} \label{Definition 9}
(\textbf{Unilateral $Z$-transform} \cite{Zwiki})
The unilateral $Z$-transform of a discrete time signal $x(k)$, $k\ge0$ is the formal power series $X(z)$ defined as
\begin{equation}\label{Eq11}
{\textstyle{} {\mathcal Z} \left\{ {x\left( {k} \right)} \right\} \triangleq \sum\nolimits_{k = 0}^{ + \infty } { z^{-k}x\left( {k} \right) } ,}
\end{equation}
where $z$ is a complex number.
\end{definition}

\begin{definition} \label{Definition 10}
(\textbf{Sampling based $N$-transform} \cite{Cheng:2011Book})
The $N$-transform of a sampling based time signal $x(t)$, $t\in\mathbb{N}_{a,h}$ is the formal power series $X(s)$ defined as
\begin{equation}\label{Eq12}
{\textstyle {{\mathcal N}_{a,h}}\left\{ {x\left( k \right)} \right\} \triangleq \sum\nolimits_{k = a/h}^{ + \infty } {{{\left( {1 - sh} \right)}^{k - 1}}x\left( {kh} \right)h}, }
\end{equation}
where $s$ is a complex number and $\mathbb{N}_{a,h}\triangleq\{a,a+h,a+2h,\cdots\}$.
\end{definition}

Considering that the Laplace transform and the $Z$-transform cannot deal with the nonzero initial instant problem directly, introducing a free variable $a\in\mathbb{R}$ brings in two new generalized transforms
\begin{equation}\label{Eq13}
{\textstyle {{\mathcal L}_a}\left\{ {x\left( t \right)} \right\} \triangleq \int_a^{ + \infty } {{{\rm{e}}^{ - s\left( {t - a} \right)}}x\left( t \right){\rm{d}}t} ,}
\end{equation}
\begin{equation}\label{Eq14}
{\textstyle{} {\mathcal Z}_a \left\{ {x\left( {k} \right)} \right\} \triangleq \sum\nolimits_{k = 0}^{ + \infty } {z^{-k}x\left( {k+a} \right) } , }
\end{equation}
and their inverse transforms
\begin{equation}\label{Eq15}
{\textstyle {{\mathcal L}_a^{-1}}\left\{ {X\left( s \right)} \right\}  \triangleq \frac{1}{{2\pi {\rm{i}}}}\int_{\beta  - {\rm{i}}\infty }^{\beta  + {\rm{i}}\infty } { {{\rm{e}}^{s\left( {t - a} \right)}}{X\left( s \right)}{\rm{d}}s},}
\end{equation}
\begin{equation}\label{Eq16}
{\textstyle {\mathcal Z}_a^{ - 1}\left\{ {X\left( z \right)} \right\} \triangleq \frac{1}{{2\pi {\rm{i}}}}\oint_c {{z^{k - a - 1}}X\left( z \right){\rm{d}}z} ,}
\end{equation}
where $\beta$ is a real number so that the contour path of integration is in the region of convergence of $X(s)$ and $c$ is a closed curve rotating around the point $(0,{\rm i}0)$ anticlockwise located in the convergent domain of $X(z)$ \cite{Wei:2019ISA,Wei:2019JCND}.

By applying the $Z$-transform for fractional order problem, the binomial expression with infinite terms will emerge. As a result, the sampling based $N$-transform becomes a good choice. Furthermore, due to the fact $\mathop {\lim }\limits_{h \to 0} {\left( {1 - sh} \right)^{1/h}} = {{\rm{e}}^{ - s}}$, when $a=0$, the so-called sampling based $N$-transform degenerates into the classical Laplace transform. For a special case, i.e., the sampling period $h=1$, one obtains the following transform.

\begin{definition} \label{Definition 11}
(\textbf{Sampling free $N$-transform} \cite{Atici:2009EJQTDE})
For $x:\mathbb{N}_{a+1} \to \mathbb{R}$, $a\in\mathbb{R}$, the $N$-transform of $x(k)$ is defined by
\begin{equation}\label{Eq17}
{\textstyle {{\mathcal N}_a}\left\{ {x\left( k \right)} \right\} \triangleq \sum\nolimits_{k = 1}^{ + \infty } {{{\left( {1 - s} \right)}^{k - 1}}x\left( {k + a} \right)} }
\end{equation}
for those values of $s$ such that the infinite series converges.
\end{definition}

Note that the so-called $N$-transform is the abbreviation of the nabla Laplace transform. Without special instructions, the $N$-transform adopted in the subsequent discussion indicates the sampling free $N$-transform.

\begin{definition} \label{Definition 12}
(\textbf{Region of convergence (ROC)})
For $x:\mathbb{N}_{a+1} \to \mathbb{R}$, $a\in\mathbb{R}$, the region of convergence for $X(s)={{\mathcal N}_a}\left\{ {x\left( k \right)} \right\} $ is defined as the set of points in the complex plane for which the infinite series converges.
\begin{equation}\label{Eq18}
{\textstyle {\rm{ROC}} = \left\{ {s:\big| {\sum\nolimits_{k = 1}^{ + \infty } {{{\left( {1 - s} \right)}^{k - 1}}x\left( {k + a} \right)} } \big| <  + \infty } \right\} .}
\end{equation}
\end{definition}
\section{MAIN RESULTS}\label{Section 3}
In this section, the essential properties of the $N$-transform will be reviewed and investigated.
\begin{lemma}\label{Lemma 1}
(\textbf{Existence} \cite{Goodrich:2015Book})
For $x:\mathbb{N}_{a+1} \to \mathbb{R}$, $a\in\mathbb{R}$, if there exist numbers $M > 0$, $r>0$ and $T \in \mathbb{N}_{a+1}$ such that $|f(k)|\le M r^k$ for all $t\in \mathbb{N}_{T}$, then its $N$-transform exists for $|s-1| < \frac{1}{r}$.
\end{lemma}

\begin{lemma}\label{Lemma 2}
(\textbf{Uniqueness} \cite{Goodrich:2015Book})
For $x,y:\mathbb{N}_{a+1} \to \mathbb{R}$, $a\in\mathbb{R}$, $x(k)=y(k)$, $k\in\mathbb{N}_{a+1}$, if and only if there exist a constant $r>0$ satisfying ${\mathcal N}_a\{x(k)\}={\mathcal N}_a\{y(k)\}$ for $|s-1|<r$.
\end{lemma}

\begin{theorem}\label{Theorem 1}
(\textbf{Linearity})
For $x,y:\mathbb{N}_{a+1} \to \mathbb{R}$, $a\in\mathbb{R}$, if $X(s)={{\mathcal N}_a}\left\{ {x(k)} \right\}$ and $Y(s)={{\mathcal N}_a}\left\{ {y(k)} \right\}$ converge for $|s-1| < r_1$ and $|s-1| < r_2$ respectively, where $r_1,r_2 > 0$, then for any $c,d\in\mathbb{R}$, one has ${{\mathcal N}_a}\left\{ {cx(k) + dy(k)} \right\} = cX(s) + dY(s)$ with $|s-1| <\min\{ r_1,r_2\}$.
\end{theorem}
\begin{proof} By using Definition \ref{Definition 11}, it follows
\begin{equation}\label{Eq19}
{\textstyle \begin{array}{l}
{{\mathcal N}_a}\left\{ {cx(k) + dy(k)} \right\}\\
 = \sum\nolimits_{k = 1}^{ + \infty } {{{\left( {1 - s} \right)}^{k - 1}}\left[ {cx(k + a) + dy(k + a)} \right]} \\
 = c\sum\nolimits_{k = 1}^{ + \infty } {{{\left( {1 - s} \right)}^{k - 1}}x\left( {k + a} \right)}  + d\sum\nolimits_{k = 1}^{ + \infty } {{{\left( {1 - s} \right)}^{k - 1}}y\left( {k + a} \right)} \\
 = cX(s) + dY(s).
\end{array}}
\end{equation}
\end{proof}

\begin{theorem}\label{Theorem 2}
(\textbf{Time advance})
For $x:\mathbb{N}_{a+1} \to \mathbb{R}$, $a\in\mathbb{R}$, if $X(s)={{\mathcal N}_a}\left\{ {x(k)} \right\}$ converges for $|s-1| < r$ for some $r > 0$, then for any $m\in\mathbb{N}$, one has ${{\mathcal N}_a}\left\{ {x(k{\rm{ + }}m)} \right\} = {\left( {1 - s} \right)^{ - m}}$ $\big[ {X\left( s \right) - \sum\nolimits_{j = 1}^m {{{\left( {1 - s} \right)}^{j - 1}}x\left( {j + a} \right)} } \big]$ with $|s-1| < r$.
\end{theorem}
\begin{proof}
The expected result can be obtained as follows
\begin{equation}\label{Eq20}
\begin{array}{l}
{{\cal N}_a}\left\{ {x(k{\rm{ + }}m)} \right\}\\
= \sum\nolimits_{k = 1}^{ + \infty } {{{\left( {1 - s} \right)}^{k - 1}}x\left( {k + m + a} \right)} \\
={\left( {1 - s} \right)^{ - m}}\sum\nolimits_{k = 1}^{ + \infty } {{{\left( {1 - s} \right)}^{k + m - 1}}x\left( {k + m + a} \right)} \\
={\left( {1 - s} \right)^{ - m}}\sum\nolimits_{j = 1 + m}^{ + \infty } {{{\left( {1 - s} \right)}^{j - 1}}x\left( {j + a} \right)} \\
={\left( {1 - s} \right)^{ - m}}\big[ \sum\nolimits_{j = 1}^{ + \infty } {{{\left( {1 - s} \right)}^{j - 1}}x\left( {j + a} \right)} \\
\hspace{56pt}- \sum\nolimits_{j = 1}^m {{{\left( {1 - s} \right)}^{j - 1}}x\left( {j + a} \right)}  \big]\\
={\left( {1 - s} \right)^{ - m}}\big[ {X\left( s \right) - \sum\nolimits_{j = 1}^m {{{\left( {1 - s} \right)}^{j - 1}}x\left( {j + a} \right)} } \big].
\end{array}
\end{equation}
\end{proof}

\begin{theorem}\label{Theorem 3}
(\textbf{Time delay})
For $x:\mathbb{N}_{a+1} \to \mathbb{R}$, $a\in\mathbb{R}$, if $X(s)={{\mathcal N}_a}\left\{ {x(k)} \right\}$ converges for $|s-1| < r$ for some $r > 0$, then for any $m\in\mathbb{N}$, one has ${{\cal N}_a}\left\{ {x(k - m)} \right\} = {\left( {1 - s} \right)^m}X\left( s \right)$ with $|s-1| < r$.
\end{theorem}
\begin{proof}
Definition \ref{Definition 11} leads to
\begin{equation}\label{Eq21}
\begin{array}{l}
{{\mathcal N}_a}\left\{ {x(k - m)} \right\}\\
= \sum\nolimits_{k = 1}^{ + \infty } {{{\left( {1 - s} \right)}^{k - 1}}x\left( {k - m + a} \right)} \\
={\left( {1 - s} \right)^m}\sum\nolimits_{k = 1}^{ + \infty } {{{\left( {1 - s} \right)}^{k - m - 1}}x\left( {k - m + a} \right)}.
\end{array}
\end{equation}

Defining $j = k - m$, it follows
\begin{equation}\label{Eq22}
\begin{array}{l}
{{\cal N}_a}\left\{ {x(k - m)} \right\} \\
={\left( {1 - s} \right)^m}\sum\nolimits_{j = 1 - m}^{ + \infty } {{{\left( {1 - s} \right)}^{j - 1}}x\left( {j + a} \right)} \\
={\left( {1 - s} \right)^m}\big[ \sum\nolimits_{j = 1}^{ + \infty } {{{\left( {1 - s} \right)}^{j - 1}}x\left( {j + a} \right)}\\
  \hspace{50pt}+ \sum\nolimits_{j = 1 - m}^0 {{{\left( {1 - s} \right)}^{j - 1}}x\left( {j + a} \right)}  \big]\\
={\left( {1 - s} \right)^m}\big[ {X\left( s \right) + \sum\nolimits_{j = 1 - m}^0 {{{\left( {1 - s} \right)}^{j - 1}}x\left( {j + a} \right)} } \big].
\end{array}
\end{equation}

For any $k \le a$, $x\left( k \right) = 0$ and therefore the desired result ${{\cal N}_a}\left\{ {x(k - m)} \right\} = {\left( {1 - s} \right)^m}X\left( s \right)$ follows.
\end{proof}

Theorems \ref{Theorem 1}-\ref{Theorem 3} are similar to Theorem 2.2 and Lemma 2.3 in \cite{Mohan:2014CMS}, while the transforms are different. To be more exact, Theorem 2.2 is a special case of Theorem \ref{Theorem 1} and Lemma 2.3 is the special case of Theorem \ref{Theorem 1} and Theorem \ref{Theorem 2}.

\begin{lemma}\label{Lemma 3}
(\textbf{Right shifting} \cite{Ahrendt:2012CAA})
For $x:\mathbb{N}_{a+1} \to \mathbb{R}$,  $a\in\mathbb{R}$, if $X(s)={{\mathcal N}_a}\left\{ {x(k)} \right\}$ converges for $|s-1| < r$ for some $r > 0$, then for any $m\in\mathbb{N}$, one has ${{\mathcal N}_{a + m}}\left\{ {x\left( k \right)} \right\} = {\left( {1 - s} \right)^{ - m}}X\left( s \right) - \sum\nolimits_{k = 1}^n {{{\left( {1 - s} \right)}^{m - k + 1}}x\left( {k + a} \right)} $ with $|s-1|<r$.
\end{lemma}

\begin{theorem}\label{Theorem 4}
(\textbf{Left shifting})
For $x:\mathbb{N}_{a+1} \to \mathbb{R}$,  $a\in\mathbb{R}$, if $X(s)={{\mathcal N}_a}\left\{ {x(k)} \right\}$ converges for $|s-1| < r$ for some $r > 0$, then for any $m\in\mathbb{N}$, one has ${{\mathcal N}_{a - m}}\left\{ {x\left( k \right)} \right\} = {\left( {1 - s} \right)^{  m}}X\left( s \right)$ with $|s-1|<r$.
\end{theorem}
\begin{proof}
According to Definition \ref{Definition 11}, it follows
\begin{equation}\label{Eq23}
\begin{array}{l}
{{\mathcal N}_{a - m}}\left\{ {x\left( k \right)} \right\}\\
 = \sum\nolimits_{k = 1}^{ + \infty } {{{\left( {1 - s} \right)}^{k + 1}}x\left( {k + a - m} \right)} \\
 = {\left( {1 - s} \right)^m}\sum\nolimits_{j = 1 - m}^{ + \infty } {{{\left( {1 - s} \right)}^{j + 1}}x\left( {j + a} \right)} \\
 = {\left( {1 - s} \right)^m}\sum\nolimits_{j = 1}^{ + \infty } {{{\left( {1 - s} \right)}^{j + 1}}x\left( {j + a} \right)} \\
 \hspace{10pt}+ {\left( {1 - s} \right)^m}\sum\nolimits_{j = 1 - m}^0 {{{\left( {1 - s} \right)}^{j + 1}}x\left( {j + a} \right)} \\
 ={\left( {1 - s} \right)^m}\big[ {X\left( s \right) + \sum\nolimits_{j = 1 - m}^0 {{{\left( {1 - s} \right)}^{j - 1}}x\left( {j + a} \right)} } \big].
\end{array}
\end{equation}

For any $k \le a$, $x\left( k \right) = 0$ and therefore the desired result ${{\mathcal N}_{a - m}}\left\{ {x\left( k \right)} \right\} = {\left( {1 - s} \right)^{  m}}X\left( s \right)$ follows.
\end{proof}

In fact, Theorem \ref{Theorem 2} and Lemma \ref{Lemma 3} discuss the same problem. Theorem \ref{Theorem 3} and Theorem \ref{Theorem 4} reflect the same property.

\begin{lemma} \label{Lemma 4}
(\textbf{Generalized rising function} \cite{Goodrich:2015Book})
For any $\alpha\in\mathbb{C}$ and $\alpha\notin\mathbb{Z}_-$, one has ${{\mathcal N}_a}\left\{ {\left( {k - a} \right)\overline {^\alpha } } \right\} = \frac{{\Gamma \left( {\alpha  + 1} \right)}}{{{s^{\alpha  + 1}}}}$ for $|s-1|<1$, where the rising function is defined by $p\overline {^q}=\frac{\Gamma(p+q)}{\Gamma(p)}$.
\end{lemma}

\begin{lemma}\label{Lemma 5}
(\textbf{Discrete Mittag--Leffler function} \cite{Wei:2019CNSNSc})
For the discrete Mittag--Leffler function defined as
${{\mathcal F}_{\alpha ,\beta }}\left( {\lambda ,k,a} \right) \triangleq$ $  \sum\nolimits_{j = 0}^{+ \infty}  {\frac{{{\lambda ^j}}}{{\Gamma \left( {j\alpha  + \beta } \right)}} \left( k-a\right)^{\overline{j\alpha  + \beta -1}}}$, one has ${\mathcal N}_a\big\{ {{{\mathcal F}_{\alpha ,\beta }}\left( {\lambda ,k,a} \right)} \big\} = \frac{{{s^{\alpha  - \beta }}}}{{{s^\alpha } - \lambda }}$ with $\left|\lambda  \right| < \left| s \right|^\alpha$ and $|s-1|<1$.
\end{lemma}

Correspondingly, for the continuous Mittag--Leffler function ${{\mathcal E}_{\alpha ,\beta }}\left( {\lambda ,t,a} \right)$ $ \triangleq \sum\nolimits_{j = 0}^{+ \infty}  {\frac{{{\lambda ^j}}}{{\Gamma \left( {j\alpha  + \beta } \right)}} \left( t-a\right)^{{j\alpha }}}$, its Laplace transform satisfies ${\mathcal L}_a\big\{ {(t-a)^{\beta-1}{{\mathcal E}_{\alpha ,\beta }}\left( {\lambda, t,a} \right)} \big\} = \frac{{{s^{\alpha  - \beta }}}}{{{s^\alpha } - \lambda }}$ with $\left|\lambda  \right| < \left| s \right|^\alpha$ \cite{Monje:2010Book}.

\begin{theorem}\label{Theorem 5}
(\textbf{Scaling in the frequency domain})
For $x:\mathbb{N}_{a+1} \to \mathbb{R}$, $a\in\mathbb{R}$, if $X(s)={{\mathcal N}_a}\left\{ {x(k)} \right\}$, then one has ${{\mathcal N}_a}\big\{ {{{\left( {1 + \lambda } \right)}^{a + 1 - k}}x(k)} \big\} = X\big( {\frac{{s + \lambda }}{{1 + \lambda }}} \big)$ for any $\lambda\neq-1$.
\end{theorem}
\begin{proof}
By applying the basic definition, it obtains
\begin{equation}\label{Eq24}
\begin{array}{l}
{{\mathcal N}_a}\big\{ {{{\left( {1 + \lambda } \right)}^{a + 1 - k}}x(k)} \big\}\\
 = \sum\nolimits_{k = 1}^{ + \infty } {{{\left( {1 - s} \right)}^{k - 1}}{{\left( {1 + \lambda } \right)}^{1 - k}}x\left( {k + a} \right)} \\
 = \sum\nolimits_{k = 1}^{ + \infty } {{{\big( {1 - \frac{{s + \lambda }}{{1 + \lambda }}} \big)}^{k - 1}}x\left( {k + a} \right)} \\
 = X\big( {\frac{{s + \lambda }}{{1 + \lambda }}} \big).
\end{array}
\end{equation}
\end{proof}

Notably, when $\lambda  = 0$, $\frac{{s + \lambda }}{{1 + \lambda }} = s$ and the original case appears. When $\lambda  \to  + \infty$, $\frac{{s + \lambda }}{{1 + \lambda }} \to 1$ and $x\left( {a + 1} \right) = X\left( 1 \right)$. Actually, Theorem \ref{Theorem 5} is implied in Theorem 5 of \cite{Wei:2019CNSNSb}.

\begin{theorem}\label{Theorem 6}
(\textbf{Complex conjugation})
For $x:\mathbb{N}_{a+1} \to \mathbb{R}$, $a\in\mathbb{R}$, if $X(s)={{\mathcal N}_a}\left\{ {x(k)} \right\}$ converges for $|s-1| < r$ for some $r > 0$, then one has ${{\mathcal N}_a}\left\{ {x{^*}(k)} \right\} = X{^*}(s{^*})$ for $|s-1| < r$, where $\zeta^*$ stands for the conjugate transpose of $\zeta$.
\end{theorem}
\begin{proof}
The primitive operation results in
\begin{equation}\label{Eq25}
\begin{array}{rl}
{{\mathcal N}_a}\left\{ {x{^*}(k)} \right\}& = \sum\nolimits_{k = 1}^{ + \infty } {{{\left( {1 - s} \right)}^{k - 1}}x{^*}\left( {k + a} \right)} \\
 &= \sum\nolimits_{k = 1}^{ + \infty } {[ {{{\left( {1 - s{^*}} \right)}^{k - 1}}x\left( {k + a} \right)} ]{^*}} \\
 &= \big[ {\sum\nolimits_{k = 1}^{ + \infty } {{{\left( {1 - s{^*}} \right)}^{k - 1}}x\left( {k + a} \right)} } \big]{^*}\\
 &= X{^*}(s{^*}).
\end{array}
\end{equation}
\end{proof}

\begin{theorem}\label{Theorem 7}
(\textbf{Real part})
For $x:\mathbb{N}_{a+1} \to \mathbb{R}$,  $a\in\mathbb{R}$, if $X(s)={{\mathcal N}_a}\left\{ {x(k)} \right\}$ converges for $|s-1| < r$ for some $r > 0$, then one has ${{\mathcal N}_a}\left\{ {{\mathop{\rm Re}\nolimits} \left\{ {x(k)} \right\}} \right\} = \frac{1}{2}\left[ {X( s ) + X{^*}( {s{^*}} )} \right]$ for $|s-1| < r$, where ${\rm Re}\{\cdot\}$ represents the real part.
\end{theorem}
\begin{proof}
By adopting the linearity and the complex conjugation property, one has
\begin{equation}\label{Eq26}
\begin{array}{rl}
{{\mathcal N}_a}\left\{ {{\mathop{\rm Re}\nolimits} \left\{ {x(k)} \right\}} \right\}& = {{\mathcal N}_a}\left\{ {\frac{1}{2}\left[ {x\left( k \right) + x{^*}\left( k \right)} \right]} \right\}\\
 &= \frac{1}{2}{{\mathcal N}_a}\left\{ {x\left( k \right)} \right\} + \frac{1}{2}{{\mathcal N}_a}\left\{ {x{^*}\left( k \right)} \right\}\\
 &= \frac{1}{2}\left[ {X( s ) + X{^*}( {s{^*}} )} \right].
\end{array}
\end{equation}
\end{proof}

\begin{theorem}\label{Theorem 8}
(\textbf{Imaginary part})
For $x:\mathbb{N}_{a+1} \to \mathbb{R}$, $a\in\mathbb{R}$, if $X(s)={{\mathcal N}_a}\left\{ {x(k)} \right\}$ converges for $|s-1| < r$ for some $r > 0$, then one has ${{\mathcal N}_a}\left\{ {{\mathop{\rm Im}\nolimits} \left\{ {x(k)} \right\}} \right\} = \frac{1}{{2{\rm{i}}}}\left[ {X( s ) - X{^*}( {s{^*}} )} \right]$ for $|s-1| < r$, where ${\rm Im}\{\cdot\}$ means the imaginary part.
\end{theorem}
\begin{proof}
With the similar operation in (\ref{Eq26}), one has
\begin{equation}\label{Eq27}
\begin{array}{rl}
{{\mathcal N}_a}\left\{ {{\mathop{\rm Im}\nolimits} \left\{ {x(k)} \right\}} \right\}& = {{\mathcal N}_a}\left\{ {\frac{1}{2{\rm i}}\left[ {x\left( k \right) - x{^*}\left( k \right)} \right]} \right\}\\
 &= \frac{1}{2{\rm i}}{{\mathcal N}_a}\left\{ {x\left( k \right)} \right\} - \frac{1}{2{\rm i}}{{\mathcal N}_a}\left\{ {x{^*}\left( k \right)} \right\}\\
 &= \frac{1}{2{\rm i}}\left[ {X( s ) - X{^*}( {s{^*}} )} \right].
\end{array}
\end{equation}
\end{proof}

\begin{theorem}\label{Theorem 9}
(\textbf{Differentiation in the frequency domain})
For $x:\mathbb{N}_{a+1} \to \mathbb{R}$,  $a\in\mathbb{R}$, if $X(s)={{\mathcal N}_a}\left\{ {x(k)} \right\}$, then one has ${{\mathcal N}_a}\left\{ \left( {k - a - 1} \right)x(k) \right\}= - \left( {1 - s} \right)\frac{{{\rm{d}}X(s)}}{{{\rm{d}}s}}$.
\end{theorem}
\begin{proof} From formula (\ref{Eq17}), one has
\begin{equation}\label{Eq28}
\begin{array}{rl}
\frac{{{\rm{d}}X(s)}}{{{\rm{d}}s}}&=\sum\nolimits_{k = 1}^{ + \infty } {\frac{{\rm{d}}}{{{\rm{d}}s}}{{\left( {1 - s} \right)}^{k - 1}}x\left( {k + a} \right)} \\
&= - \sum\nolimits_{k = 1}^{ + \infty } {\left( {k - 1} \right){{\left( {1 - s} \right)}^{k - 2}}x\left( {k + a} \right)} \\
&=  - {\left( {1 - s} \right)^{ - 1}}\sum\nolimits_{k = 1}^{ + \infty } {{{\left( {1 - s} \right)}^{k - 1}}\left( {k - 1} \right)x\left( {k + a} \right)} \\
&= - {\left( {1 - s} \right)^{ - 1}}{{\mathcal N}_a}\left\{ {\left( {k - a - 1} \right)x\left( k \right)} \right\},
\end{array}
\end{equation}
which is equivalent to the target formula.
\end{proof}

Taking this operation repeatedly, then for any $m\in\mathbb{Z}_+$, it follows
${{\mathcal N}_a}\left\{ \left( {k - a - 1} \right)^mx(k) \right\}= {\big[ { - \left( {1 - s} \right)\frac{{\rm{d}}}{{{\rm{d}}s}}} \big]^m}X(s)$.

\begin{theorem}\label{Theorem 10}
(\textbf{Integration in the frequency domain})
For $x:\mathbb{N}_{a+1} \to \mathbb{R}$,  $a\in\mathbb{R}$, if $X(s)={{\mathcal N}_a}\left\{ {x(k)} \right\}$, then for any $m\in\mathbb{N}$, one has ${{\mathcal N}_a}\big\{{\left( {k - a - m - 1} \right)^{ - 1}}x(k) \big\}=  - {\left( {1 - s} \right)^m}\int_1^s {{{\left( {1 - \eta } \right)}^{ - m - 1}}X(\eta ){\rm{d}}\eta } $.
\end{theorem}
\begin{proof} From formula (\ref{Eq17}), it obtains
\begin{equation}\label{Eq29}
\begin{array}{l}
{\left( {1 - s} \right)^m}\int_1^s {{{\left( {1 - \eta } \right)}^{ - m - 1}}X(\eta ){\rm{d}}\eta } \\
 = {\left( {1 - s} \right)^m}\int_1^s {{{\left( {1 - \eta } \right)}^{ - m - 1}}\sum\nolimits_{k = 1}^{ + \infty } {{{\left( {1 - \eta } \right)}^{k - 1}}x\left( {k + a} \right)} {\rm{d}}\eta } \\
 = {\left( {1 - s} \right)^m}\sum\nolimits_{k = 1}^{ + \infty } {\int_1^s {{{\left( {1 - \eta } \right)}^{k - m - 2}}{\rm{d}}\eta } x\left( {k + a} \right)} \\
 =  - {\left( {1 - s} \right)^m}\sum\nolimits_{k = 1}^{ + \infty } { {{{\left( {k - m - 1} \right)}^{ - 1}}{{\left( {1 - \eta } \right)}^{k - m - 1}}} \big|_1^sx\left( {k + a} \right)} \\
 =  - \sum\nolimits_{k = 1}^{ + \infty } {{{\left( {1 - s} \right)}^{k - 1}}{{\left( {k - m - 1} \right)}^{ - 1}}x\left( {k + a} \right)} \\
 =  - {{\mathcal N}_a}\big\{ {{{\left( {k - a - m - 1} \right)}^{ - 1}}x\left( k \right)} \big\}.
\end{array}
\end{equation}
\end{proof}
When $m=0$, ${{\mathcal N}_a}\big\{{\left( {k - a  - 1} \right)^{ - 1}}x(k) \big\}=  - \int_1^s {{{\left( {1 - \eta } \right)}^{  - 1}}X(\eta ){\rm{d}}\eta } $ follows immediately. Furthermore, for both Theorem \ref{Theorem 9} and Theorem \ref{Theorem 10}, the results can be extended to the fractional order case with minor changes.

\begin{theorem}\label{Theorem 11}
(\textbf{Accumulation})
For $x:\mathbb{N}_{a+1} \to \mathbb{R}$,  $a\in\mathbb{R}$, if $X(s)={{\mathcal N}_a}\left\{ {x(k)} \right\}$ converges for $|s-1| < r$ for some $r > 0$, then one has ${{\mathcal N}_a}\big\{\sum\nolimits_{j = a + 1}^k {x(j)}  \big\}= \frac{1}{s}X(s)$ for $|s-1|<\min\{1,r\}$.
\end{theorem}
\begin{proof} According to Definition \ref{Definition 11}, one has
\begin{equation}\label{Eq30}
\begin{array}{l}
\sum\nolimits_{k = 1}^{ + \infty } {{{\left( {1 - s} \right)}^{k - 1}}\sum\nolimits_{j = a + 1}^{k + a} {x(j)} } \\
 = \sum\nolimits_{k = 1}^{ + \infty } {{{\left( {1 - s} \right)}^{k - 1}}\sum\nolimits_{j = 1}^k {x(j + a)} } \\
 = \sum\nolimits_{j = 1}^{ + \infty } {x(j + a)\sum\nolimits_{k = j}^{ + \infty } {{{\left( {1 - s} \right)}^{k - 1}}} } \\
 = \sum\nolimits_{j = 1}^{ + \infty } {x(j + a)\frac{{{{\left( {1 - s} \right)}^{j - 1}}}}{s}} \\
 = \frac{1}{s}X(s).
\end{array}
\end{equation}
\end{proof}

\begin{lemma}\label{Lemma 6}
(\textbf{Convolution} \cite{Goodrich:2015Book})
For $x,y:\mathbb{N}_{a+1} \to \mathbb{R}$, $a\in\mathbb{R}$, if $X(s)={{\mathcal N}_a}\left\{ {x(k)} \right\}$ and $Y(s)={{\mathcal N}_a}\left\{ {y(k)} \right\}$, then one has ${{\mathcal N}_a}\left\{ {x(k)\ast y(k)} \right\} = X(s)Y(s)$, where $\ast$ indicates the convolution operation, i.e., $x(k)\ast y(k)\triangleq\sum\nolimits_{j = a + 1}^k {x\left( {k - j + a + 1} \right)y\left( j \right)} $.
\end{lemma}

\begin{lemma}\label{Lemma 7}
(\textbf{Inverse $N$-transform} \cite{Wei:2019AJC})
For $x:\mathbb{N}_{a+1} \to \mathbb{R}$,  $a\in\mathbb{R}$, if $X(s)={{\mathcal N}_a}\left\{ {x(k)} \right\}$, then the inverse nabla Laplace transform can be expressed as $x\left( k \right) = {\mathcal N}_a^{ - 1}\{ X\left( s \right)\}  \buildrel \Delta \over = \frac{1}{{2\pi {\rm{i}}}}\oint_c {X\left( s \right){{(1 - s)}^{ - k + a}}{\rm{d}}s} ,k \in {\mathbb{N}_{a + 1}}$, where $c$ is a closed curve rotating around the point $(1,{\rm i}0)$ clockwise and it also locates in the convergent domain of $X(s)$.
\end{lemma}

\begin{theorem}\label{Theorem 12}
(\textbf{Multiplication})
For $x,y:\mathbb{N}_{a+1} \to \mathbb{R}$, $a\in\mathbb{R}$, if $X(s)={{\mathcal N}_a}\left\{ {x(k)} \right\}$ and $Y(s)={{\mathcal N}_a}\left\{ {y(k)} \right\}$, then one has ${{\cal N}_a}\left\{ {x(k)y(k)} \right\} = \frac{1}{{2\pi {\rm{i}}}}\oint_c {X\left( z \right)Y\big( {\frac{{s - z}}{{1 - z}}} \big){{(1 - z)}^{ - 1}}{\rm{d}}z} $, where $c$  is a closed curve rotating around the point $(1,{\rm i}0)$ clockwise located in the overlapping convergent domain of $X(z)$ and $Y\big( {\frac{{s - z}}{{1 - z}}} \big)$.
\end{theorem}
\begin{proof}
By applying Lemma \ref{Lemma 7}, one has
\begin{equation}\label{Eq31}
\begin{array}{l}
{{\mathcal N}_a}\left\{ {x(k)y(k)} \right\}\\
 = \sum\nolimits_{k = 1}^{ + \infty } {{{\left( {1 - s} \right)}^{k - 1}}x\left( {k + a} \right)} y\left( {k + a} \right)\\
 = \sum\nolimits_{k = 1}^{ + \infty } {{{\left( {1 - s} \right)}^{k - 1}}\big[ {\frac{1}{{2\pi {\rm{i}}}}\oint_c {X\left( z \right){{(1 - z)}^{ - k}}{\rm{d}}z} } \big]} y\left( {k + a} \right)\\
 = \frac{1}{{2\pi {\rm{i}}}}\sum\nolimits_{k = 1}^{ + \infty } {\oint_c {X\left( z \right){{(1 - z)}^{ - k}}{\rm{d}}z} {{\left( {1 - s} \right)}^{k - 1}}} y\left( {k + a} \right)\\
 = \frac{1}{{2\pi {\rm{i}}}}\oint_c {X\left( z \right)\sum\nolimits_{k = 1}^{ + \infty } {{{\left( {1 - s} \right)}^{k - 1}}} y\left( {k + a} \right){{(1 - z)}^{ - k}}{\rm{d}}z} \\
 = \frac{1}{{2\pi {\rm{i}}}}\oint_c {X\left( z \right)\sum\nolimits_{k = 1}^{ + \infty } {{{\big( {1 - \frac{{s - z}}{{1 - z}}} \big)}^{k - 1}}} y\left( {k + a} \right){{(1 - z)}^{ - 1}}{\rm{d}}z} \\
 = \frac{1}{{2\pi {\rm{i}}}}\oint_c {X\left( z \right)Y\big( {\frac{{s - z}}{{1 - z}}} \big){{(1 - z)}^{ - 1}}{\rm{d}}z}.
\end{array}
\end{equation}
\end{proof}

Defining $z = 1 + \rho {{\rm{e}}^{{\rm i}\theta }}$ and $s = 1 + r{{\rm{e}}^{{\rm i}\varphi }}$, then a simple result ${{\cal N}_a}\left\{ {x(k)y(k)} \right\} = \frac{1}{{2\pi }}\int_{ - \pi }^\pi  {X(1 + \rho {{\rm{e}}^{{\rm i}\theta }})Y(1 - \frac{r}{\rho }{{\rm{e}}^{{\rm i}\left( {\varphi  - \theta } \right)}}){\rm{d}}\theta } $ can be obtained. Sometimes, Lemma \ref{Lemma 6} is called Convolution theorem in time domain and Theorem \ref{Theorem 12} is called Convolution theorem in frequency domain.

\begin{theorem}\label{Theorem 13}
(\textbf{Parseval's theorem})
For $x,y:\mathbb{N}_{a+1} \to \mathbb{R}$, $a\in\mathbb{R}$, if $X(s)={{\mathcal N}_a}\left\{ {x(k)} \right\}$ and $Y(s)={{\mathcal N}_a}\left\{ {y(k)} \right\}$, then one has $\sum\nolimits_{k = a + 1}^{ + \infty } {x\left( k \right)} {y{^*} }\left( k \right) = \frac{1}{{2\pi {\rm{i}}}}\oint_c {X\left( z \right){Y{^*}}\big( {\frac{{{z{^*} }}}{{{z{^*} } - 1}}} \big){{(1 - z)}^{ - 1}}{\rm{d}}z}$, where $c$  is a closed curve rotating around the point $(1,{\rm i}0)$ clockwise located in the overlapping convergent domain of $X(z)$ and ${Y{^*}}\big( {\frac{{{z{^*} }}}{{{z{^*} } - 1}}} \big)$.
\end{theorem}
\begin{proof}
Based on Definition \ref{Definition 11}, it obtains
\begin{equation}\label{Eq32}
\begin{array}{l}
\sum\nolimits_{k = a + 1}^{ + \infty } {x\left( k \right)} {y{^*} }\left( k \right)\\
 = \sum\nolimits_{k = a + 1}^{ + \infty } {\big[ {\frac{1}{{2\pi {\rm{i}}}}\oint_c {X\left( z \right){{(1 - z)}^{ - k + a}}{\rm{d}}z} } \big]} {y{^*} }\left( k \right)\\
 = \frac{1}{{2\pi {\rm{i}}}}\oint_c {X\left( z \right)\sum\nolimits_{k = a + 1}^{ + \infty } {{y{^*} }\left( k \right){{(1 - z)}^{ - k + a}}{\rm{d}}z} } \\
 = \frac{1}{{2\pi {\rm{i}}}}\oint_c {X\left( z \right)\sum\nolimits_{l = 1}^{ + \infty } {{y{^*} }\left( {l + a} \right){{(1 - z)}^{ - l}}{\rm{d}}z} } \\
 = \frac{1}{{2\pi {\rm{i}}}}\oint_c {X\left( z \right)\sum\nolimits_{l = 1}^{ + \infty } {{y{^*} }\left( {l + a} \right){{(1 - \frac{z}{{z - 1}})}^{l - 1}}{{(1 - z)}^{ - 1}}{\rm{d}}z} } \\
 = \frac{1}{{2\pi {\rm{i}}}}\oint_c {X\left( z \right){Y{^*}}\big( {\frac{{{z{^*} }}}{{{z{^*} } - 1}}} \big){{(1 - z)}^{ - 1}}{\rm{d}}z}.
\end{array}
\end{equation}
\end{proof}

Defining $z = 1 + \rho {{\rm{e}}^{{\rm{i}}\theta }}$, then the result can be simplified as $\sum\nolimits_{k = a + 1}^{ + \infty } {x\left( k \right)} {y^ * }\left( k \right) = \frac{1}{{2\pi }}\int_{ - \pi }^\pi  {X( {1 + \rho {{\rm{e}}^{{\rm{i}}\theta }}} ){Y{^*} }( {1 + {\rho ^{ - 1}}{{\rm{e}}^{{\rm{i}}\theta }}} ){\rm{d}}\theta } $. Moreover, when $x(k)=y(k)$ and $\rho=1$, the result can be simplified further $\sum\nolimits_{k = a + 1}^{ + \infty } {|x\left( k \right)|^2}  = \frac{1}{{2\pi }}\int_{ - \pi }^\pi  {|X( {1 +  {{\rm{e}}^{{\rm{i}}\theta }}} )|^2{\rm{d}}\theta } $, which means that the energy in the time domain is equal to that in the frequency domain exactly.

Lemma \ref{Lemma 6}, Theorem \ref{Theorem 12} and Theorem \ref{Theorem 13} build the bridge between the time domain and the frequency domain.

\begin{lemma}\label{Lemma 8}
(\textbf{Initial value theorem} \cite{Wei:2019CNSNSc})
For $x:\mathbb{N}_{a+1} \to \mathbb{R}$,  $a\in\mathbb{R}$, if $X(s)={{\mathcal N}_a}\left\{ {x(k)} \right\}$, then one has $x\left( a+1 \right) = \mathop {\lim }\limits_{s \to 1} X\left( s \right)$.
\end{lemma}

\begin{lemma}\label{Lemma 9}
(\textbf{Final value theorem} \cite{Wei:2019CNSNSc})
For $x:\mathbb{N}_{a+1} \to \mathbb{R}$,  $a\in\mathbb{R}$, if $X(s)={{\mathcal N}_a}\left\{ {x(k)} \right\}$ and the poles of $sX\left( s \right)$ satisfy $|s-1|>1$, then one has $ x\left( a+\infty\right) = \mathop {\lim }\limits_{s \to 0} sX\left( s \right)$.
\end{lemma}

Lemma \ref{Lemma 8} and Lemma \ref{Lemma 9} provide two effective way to access the value of $x(k)$ at $k=a+1$ and $k=a+\infty$, respectively. Along this way, a general case can be developed as \cite{Wei:2019CNSNSc}
\begin{equation}\label{Eq33}
{\textstyle f\left( {a + m } \right) = \mathop {\lim }\limits_{s \to 1} \frac{{F\left( s \right) - \sum\nolimits_{k = 1}^{m  - 1} {{{\left( {1 - s} \right)}^{k - 1}}f\left( {a + k} \right)} }}{{{{\left( {1 - s} \right)}^{m  - 1}}}},m\in\mathbb{Z}_+.}
\end{equation}

\begin{lemma}\label{Lemma 10}
(\textbf{Nabla difference} \cite{Goodrich:2015Book})
For $x:\mathbb{N}_{a+1-n} \to \mathbb{R}$,  $a\in\mathbb{R}$, if $X(s)={{\mathcal N}_a}\left\{ {x(k)} \right\}$ converges for $|s-1| < r$ for some $r > 0$, then for any $n\in\mathbb{Z}_+$, one has ${\mathcal N}_a\{ {\nabla} ^n x\left( k \right)\} =  {s^n}X(s)  - {\sum\nolimits_{j = 0}^{n - 1} s ^{n - j - 1}}\big[{\nabla^j}x\left( k \right)\big]_{k=a}$ and ${\mathcal N}_a\{ {\nabla} ^n x\left( k \right)\} =  {s^n}X(s) - {\sum\nolimits_{j = 0}^{n - 1} s ^j}\big[{\nabla^{n - j - 1}}x\left( k \right)\big]_{k=a}$ with $|s-1|<r$.
\end{lemma}

\begin{lemma}\label{Lemma 11}
(\textbf{Nabla fractional sum} \cite{Goodrich:2015Book})
For $x:\mathbb{N}_{a+1} \to \mathbb{R}$,  $a\in\mathbb{R}$, if $X(s)={{\mathcal N}_a}\left\{ {x(k)} \right\}$ converges for $|s-1| < r$ for some $r > 0$, then for any $\alpha>0$, one has ${\mathcal N}_a\{ {}_a^{}{\nabla}_k ^{-\alpha} x\left( k \right)\} =  \frac{1}{s^\alpha}X(s)$ with $|s-1|<\min\{1,r\}$.
\end{lemma}

Note that Lemma \ref{Lemma 11} actually implies the integer order case, i.e., $\alpha\in\mathbb{Z}_+$ and Theorem \ref{Theorem 11} can be regarded as a special case of Lemma \ref{Lemma 11} with $\alpha=1$.

\begin{lemma}\label{Lemma 12}
(\textbf{Nabla Caputo fractional difference} \cite{Wei:2019CNSNSb})
For $x:\mathbb{N}_{a+1-n} \to \mathbb{R}$,  $a\in\mathbb{R}$, if $X(s)={{\mathcal N}_a}\left\{ {x(k)} \right\}$ converges for $|s-1| < r$ for some $r > 0$, then for any $\alpha \in (n-1,n)$ and $n\in\mathbb{Z}_+$, one has ${\mathcal N}_a\{ {}_a^{\rm C}{\nabla}_k ^\alpha x\left( k \right)\} =  {s^\alpha}X(s)  - {\sum\nolimits_{j = 0}^{n - 1} s ^{\alpha - j - 1}}\big[{\nabla^j}x\left( k \right)\big]_{k=a}$ with $|s-1|<r$.
\end{lemma}

Interestingly, the initial conditions regarding to the Caputo definition are the same with that in Lemma \ref{Lemma 10}, i.e., $\big[{\nabla^j}x\left( k \right)\big]_{k=a}$, $j=0,1,\cdots,n-1$.

\begin{lemma}\label{Lemma 13}
(\textbf{Nabla Reimann--Liouville fractional difference} \cite{Wei:2019CNSNSb})
For $x:\mathbb{N}_{a+1-n} \to \mathbb{R}$,  $a\in\mathbb{R}$, if $X(s)={{\mathcal N}_a}\left\{ {x(k)} \right\}$ converges for $|s-1| < r$ for some $r > 0$, then for any $\alpha \in (n-1,n)$ and $n\in\mathbb{Z}_+$, one has ${\mathcal N}_a\{ {}_a^{\rm R}{\nabla}_k ^\alpha x\left( k \right)\} =  {s^\alpha}X(s) - {\sum\nolimits_{j = 0}^{n - 1} s ^j}\big[{{}_a^{\rm R}\nabla_k^{\alpha - j - 1}}x\left( k \right)\big]_{k=a}$ with $|s-1|<r$.
\end{lemma}

In contrast, the initial conditions regarding to the Reimann--Liouville definition are similar to that in Lemma \ref{Lemma 10}, i.e., $\big[{\nabla^{n - j - 1}}x\left( k \right)\big]_{k=a}$, $j=0,1,\cdots,n-1$. Let us continue to consider the case of $\alpha \in (n-1,n)$ and $n\in\mathbb{Z}_+$. Reference \cite{Wei:2019ISA} deduces the result that ${{\mathcal L}_a}\left\{ {{}_a^{\rm{C}}{\mathscr D}_t^\alpha x\left( t \right)} \right\} = {s^\alpha }X(s)- \sum\nolimits_{j = 0}^{n - 1} {{s^{\alpha  - j - 1}}\big[ {\frac{{{{\rm{d}}^j}}}{{{\rm{d}}{t^j}}}x\left( t \right)} \big]_{t = a}} $ and ${{\mathcal L}_a}\left\{ {{}_a^{\rm{R}}{\mathscr D}_t^\alpha x\left( t \right)} \right\} = {s^\alpha }X(s) - \sum\nolimits_{j = 0}^{n - 1} {{s^j}\big[{}_a^{\rm{R}}{\mathscr D}_t^{\alpha  - j - 1}x\left( t\right)\big]_{t=a}} $. It is clearly indicated that the discrete case and the continuous case share a common form.

\begin{theorem}\label{Theorem 14}
(\textbf{Nabla Gr\"{u}nwald--Letnikov fractional difference})
For $x:\mathbb{N}_{a+1} \to \mathbb{R}$,  $a\in\mathbb{R}$, if $X(s)={{\mathcal N}_a}\left\{ {x(k)} \right\}$ converges for $|s-1| < r$ for some $r > 0$, then for any $\alpha >0$ and $\alpha\notin \mathbb{Z}_+$, one has ${\mathcal N}_a\{ {}_a^{\rm G}{\nabla}_k ^\alpha x\left( k \right)\} =  {s^\alpha}X(s) $ with $|s-1|<r$.
\end{theorem}
\begin{proof}
The Gr\"{u}nwald--Letnikov fractional difference can be rewritten as
\begin{equation}\label{Eq34}
\begin{array}{rl}
{}_a^{\rm{G}}\nabla _k^\alpha x\left( k \right) &= \sum\nolimits_{j = 0}^{k - a - 1} {{{\left( { - 1} \right)}^j}\big(\begin{smallmatrix}
\alpha \\
j
\end{smallmatrix}\big)x\left( {k - j} \right)} \\
 &= \sum\nolimits_{j = 0}^{k - a - 1} {{{\left( { - 1} \right)}^j}\frac{{\Gamma \left( {\alpha  + 1} \right)}}{{\Gamma \left( {j + 1} \right)\Gamma \left( {\alpha  - j + 1} \right)}}x\left( {k - j} \right)} \\
 &= \sum\nolimits_{j = 0}^{k - a - 1} {\frac{{\Gamma \left( {j - \alpha } \right)}}{{\Gamma \left( { - \alpha } \right)\Gamma \left( {j + 1} \right)}}x\left( {k - j} \right)} \\
 &= \sum\nolimits_{j = 0}^{k - a - 1} {\frac{{\left( {j + 1} \right)\overline {^{ - \alpha  - 1}} }}{{\Gamma \left( { - \alpha } \right)}}x\left( {k - j} \right)} \\
 &= \sum\nolimits_{j = a + 1}^k {\frac{{\left( {k - j + 1} \right)\overline {^{ - \alpha  - 1}} }}{{\Gamma \left( { - \alpha } \right)}}x\left( j \right)} \\
 &= \frac{{\left( {k - a} \right)\overline {^{ - \alpha  - 1}} }}{{\Gamma \left( { - \alpha } \right)}} * x\left( k \right).
\end{array}
\end{equation}

With the adoption of Lemma \ref{Lemma 4} and Lemma \ref{Lemma 6}, the desired result can be obtained successfully.
\end{proof}

Since there exist no initial conditions in  Gr\"{u}nwald--Letnikov case, ${}_a^{\rm G}{\nabla}_k ^\alpha x\left( k+1 \right)=\lambda x\left( k \right)$ or ${}_a^{\rm G}{\nabla}_k ^\alpha x\left( k \right)=\lambda x\left( k-1 \right)$ is usually considered instead of ${}_a^{\rm G}{\nabla}_k ^\alpha x\left( k \right)=\lambda x\left( k \right)$. Combining the frequency-domain descriptions of nabla fractional difference in Lemma \ref{Lemma 12}, Lemma \ref{Lemma 13} and Theorem \ref{Theorem 14}, the following corollary can be obtained.
\begin{corollary}\label{Corollary 1}
For any $\alpha\in(n-1,n)$ and $n\in\mathbb{Z}_+$, one can conclude that
\begin{enumerate}[i)]
  \item ${}_a^{\rm{C}}\nabla _k^\alpha x\left( k \right) = v\left( k \right)$ is equivalent to ${}_a^{ }\nabla _k^{ - \alpha }v\left( k \right) = x\left( k \right)+r_1$ with $r_1=\sum\nolimits_{j = 0}^{n - 1} {\frac{{{{\left( {k - a} \right)}^j}}}{{\Gamma \left( {j + 1} \right)}}{{[{\nabla ^j}x\left( k \right)]}_{k = a}}}$;
  \item ${}_a^{\rm{R}}\nabla _k^\alpha x\left( k \right) = v\left( k \right)$ is equivalent to ${}_a^{ }\nabla _k^{ - \alpha }v\left( k \right) = x\left( k \right)+r_2$ with $r_2=\sum\nolimits_{j = 0}^{n - 1} {\frac{{{{\left( {k - a} \right)}^{\alpha  - j - 1}}}}{{\Gamma \left( {\alpha  - j} \right)}}{{[{}_a^{\rm{R}}\nabla _k^{\alpha  - j - 1}x\left( k \right)]}_{k = a}}}
$;
  \item ${}_a^{\rm{G}}\nabla _k^\alpha x\left( k \right) = v\left( k \right)$ is equivalent to ${}_a^{ }\nabla _k^{ - \alpha }v\left( k \right) = x\left( k \right)$.
\end{enumerate}
\end{corollary}

\begin{lemma}\label{Lemma 14}
(\textbf{Stable region} \cite{Wei:2019CNSNSc})
Assume $x: \mathbb{N}_{a+1}\to \mathbb{C}$, $a\in\mathbb{R}$ and $X(s)={\mathcal N}_a\{x(k)\}$. $x(k)$ is convergent with respect to $k$, i.e., $\mathop {\lim }\limits_{k \to  + \infty } x\left( k \right) = 0$, if and only if all the principal poles of $X(s)$ satisfy $|s-1|>1$.
\end{lemma}

Note that the region of convergence in Definition \ref{Definition 12} is a different concept from the stable region in Lemma \ref{Lemma 14}. The former focuses on the condition when the summation of an infinite series ${\sum\nolimits_{k = 1}^{ + \infty } {{{\left( {1 - s} \right)}^{k - 1}}x\left( {k + a} \right)} } $ is convergent. The latter concerns the poles' location of $X(s)$ which could guarantee the convergence of $x(k)$ in the time domain.

\begin{remark}\label{Remark 1}
This paper makes a short survey on the promising properties of $N$-transform. All the existing results are presented as lemma. All the newly proposed results are given in the form of theorem. Notably, among them 20 properties are proposed by the authors. Besides these elaborated properties, some interesting properties are left for the interested readers, such as the properties on time expansion, decimation, time reversal, cross-correlation, scaling in the time domain, and rotating in the frequency domain, etc \cite{Zwiki}.
\end{remark}

\section{APPLICATIONS}\label{Section 4}
In this section, four examples are provided to demonstrate the applications of $N$-transform.
\begin{example}\label{Examle 1}
Consider the discrete time unit step function $u\left( k \right) = \left\{ \begin{array}{l}
1,k \ge 0\\
0,k < 0
\end{array} \right.$. Then, for any $a\in \mathbb{R}$, one has
\begin{equation}\label{Eq35}
\begin{array}{rl}
{{\mathcal N}_a}\left\{ {u\left( {k - a - 1} \right)} \right\} &= \sum\nolimits_{k = 1}^{ + \infty } {{{\left( {1 - s} \right)}^{k - 1}}u\left( {k - 1} \right)} \\
 &= \sum\nolimits_{k = 1}^{ + \infty } {{{\left( {1 - s} \right)}^{k - 1}}} \\
 &= \frac{1}{s},
\end{array}
\end{equation}
for $|s-1|<1$. Since $u\left( {k - a - 1} \right) = {\left( {k - a} \right)^{\overline 0 }}$, with the help of Lemma \ref{Lemma 4}, the result in (\ref{Eq35}) can also be obtained.
\end{example}
\begin{example}\label{Examle 2}
Consider the discrete time unit impulse function $\delta \left( k \right) = \left\{ \begin{array}{l}
1,k = 0\\
0,k \ne 0
\end{array} \right.$. Then, for any $a\in \mathbb{R}$, one has
\begin{equation}\label{Eq36}
\begin{array}{rl}
{{\mathcal N}_a}\left\{ {\delta \left( {k - a - 1} \right)} \right\} &= \sum\nolimits_{k = 1}^{ + \infty } {{{\left( {1 - s} \right)}^{k - 1}}\delta \left( {k - 1} \right)} \\
 &= {{{{\left( {1 - s} \right)}^{k - 1}}} \big|_{k = 1}}\\
 &= 1,
\end{array}
\end{equation}
for $s\neq 1$. Since $\delta \left( {k - a - 1} \right) = \nabla u\left( {k - a - 1} \right)$, with the help of Lemma \ref{Lemma 10}, the result in (\ref{Eq36}) can also be obtained.
\end{example}

\begin{example}\label{Examle 3}
By applying the formula $\frac{1}{{{s^\alpha }}} = \int_0^{ + \infty } {\frac{{{\mu _\alpha }\left( \omega  \right)}}{{s + \omega }}{\rm{d}}\omega }$ and Theorem \ref{Theorem 5}, the system $\frac{1}{{{{\left( {\tau s + 1} \right)}^\alpha }}}$ with $\alpha\in(0,1)$ and $\tau\in(-0.5,0)\cup(0,+\infty)$ can be realized by the
linear time varying system
  \begin{equation}\label{Eq39}
\left\{ \begin{array}{l}
{\nabla ^1}\varsigma \left( {\omega ,k} \right) =  - \omega \varsigma\left( {\omega ,k} \right) + {( {1 + {\lambda ^{ - 1}}} )^{a + 1 - k}}v\left( k \right),\\
x\left( k \right) = \int_0^{ + \infty } {{\mu _\alpha }\left( \omega  \right)\kappa {{( {1 + {\tau ^{ - 1}}} )}^{a + 1 - k}}\varsigma\left( {\omega ,k} \right){\rm{d}}\omega },
\end{array} \right.
  \end{equation}
  or the linear time invariant system
  \begin{equation}\label{Eq38}
   \left\{ {\begin{array}{l}
   \tau {\nabla ^1}z\left( \omega,k \right) =  - (\omega+1) z\left( \omega,k \right) + v\left( k \right),\\
   x\left( k \right) = \int_0^{+\infty}  {{\mu _\alpha }\left(\omega\right)z\left( \omega,k \right){\rm d}\omega},
   \end{array}} \right.
  \end{equation}
  with $z\left( \omega,a \right)=0$, $\varsigma\left( \omega,a \right)=0$, ${\mu _\alpha }\left( \omega \right) = \frac{{\sin \left(\alpha \pi \right)}}{{{\omega^\alpha }\pi }}$, $\kappa  = {\left( {\tau  + 1} \right)^{ - \alpha }}$ and $\lambda  =  - 1 - \tau $.
\end{example}

\begin{example}\label{Examle 4}
Consider the system ${}_a^{\rm C}\nabla _k^\alpha x\left( k \right) = \lambda x\left( k \right)$ with $\alpha=0.5$, $x(a)=1$ and $a=3$. With the help of $N$-transform, the frequency domain form of $x(k)$ satisfies $X\left( s \right) = \frac{1}{{{s^\alpha } - \lambda }}x\left( a \right)$. When we select $\lambda=(1+\rho{\rm e}^{{\rm i}\theta})^\alpha$, from Lemma \ref{Lemma 14}, it can be concluded that if $\rho<1$, the system is unstable and vice versa. Because $x(k)$ might be a complex number as $\lambda$ is a complex number, thus we plot the real part of $x(k)$ for different $\rho$ and $\theta$ in FIGURE \ref{Fig1} - FIGURE \ref{Fig3}.
\begin{figure}[!htbp]
\centering
\includegraphics[width=0.9\columnwidth]{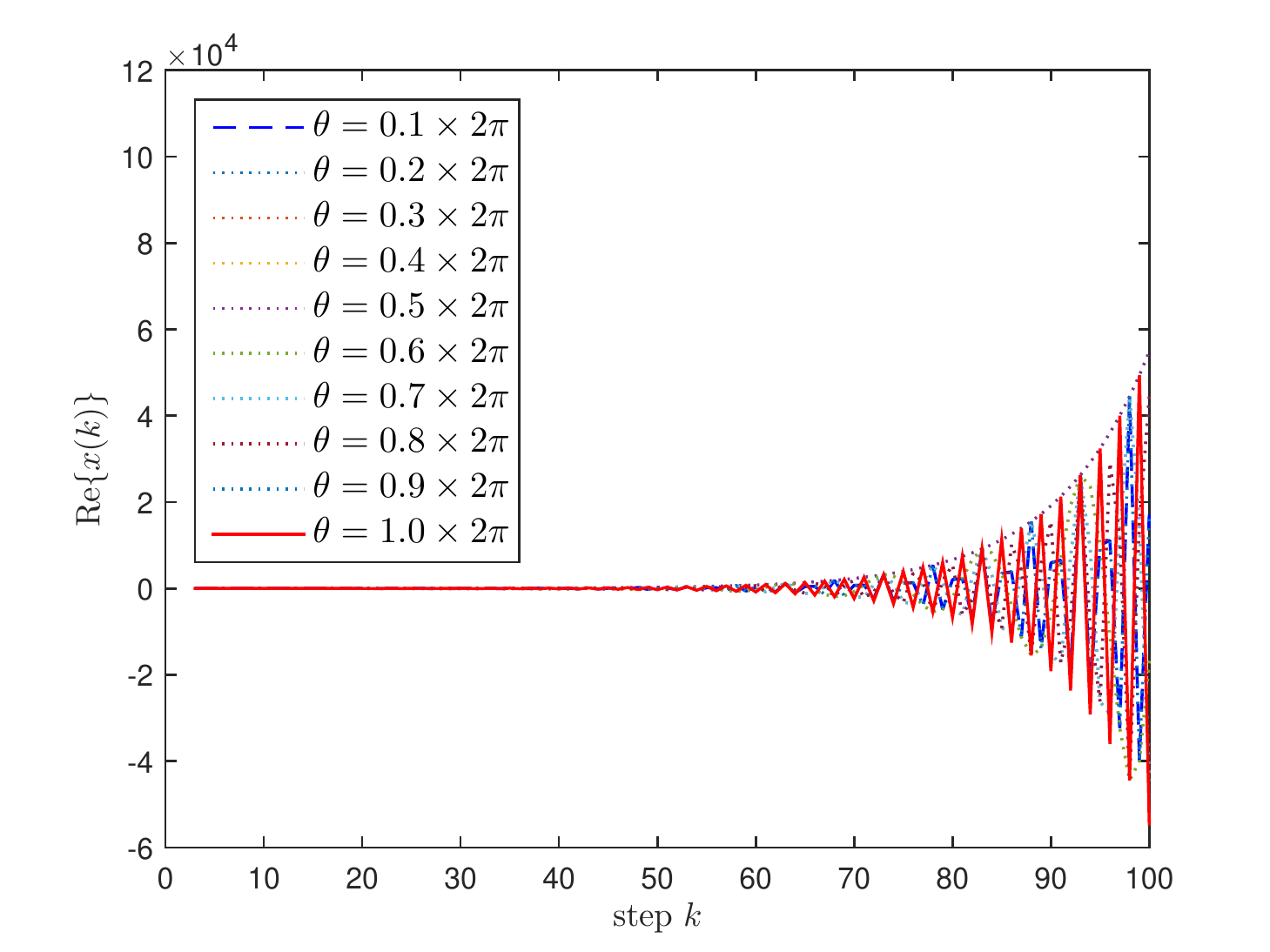}
\caption{The real part of $x(k)$ with $\rho=0.9$ and different $\theta$.}\label{Fig1}
\end{figure}
\begin{figure}[!htbp]
\centering
\includegraphics[width=0.9\columnwidth]{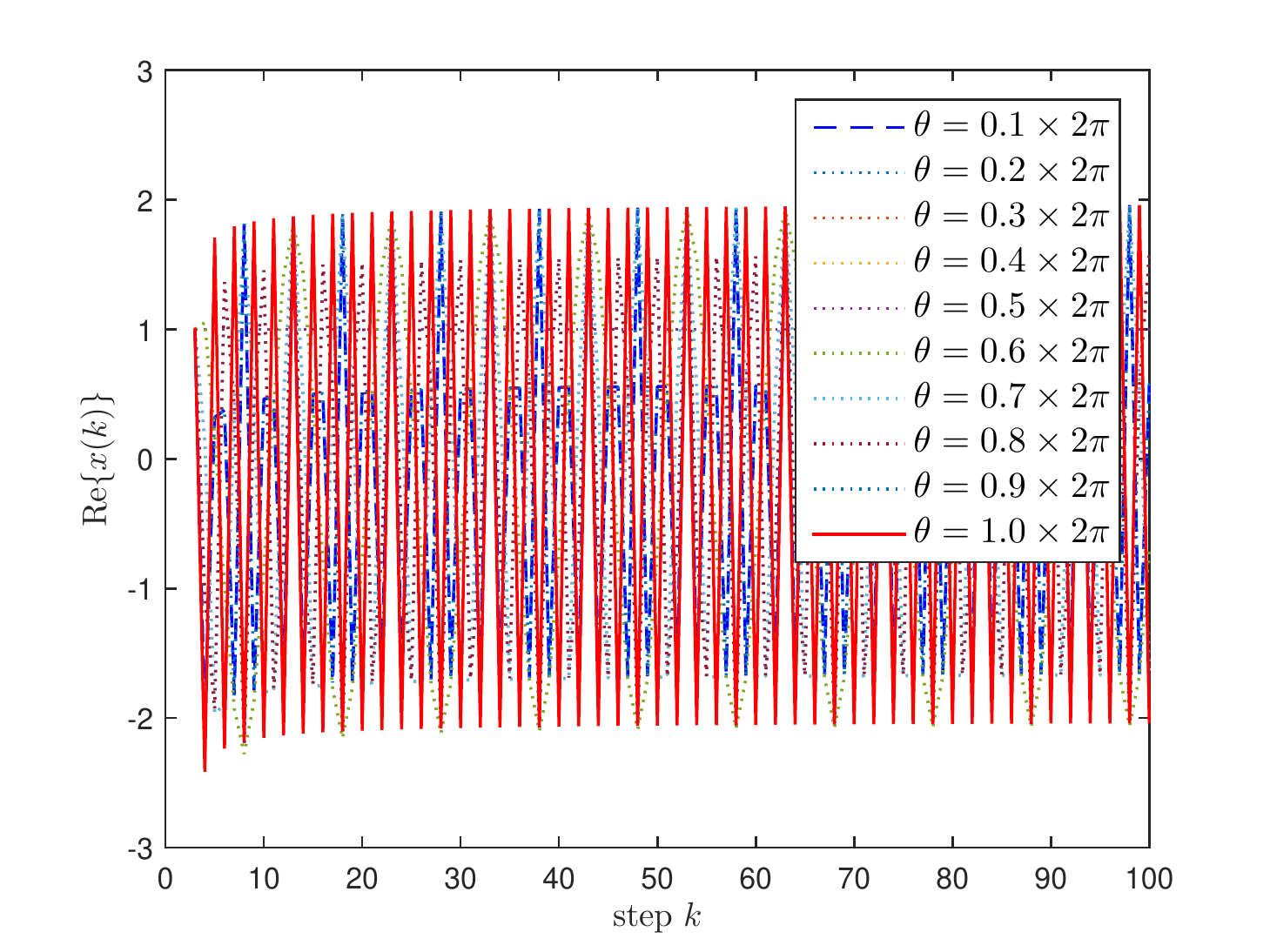}
\caption{The real part of $x(k)$ with $\rho=1.0$ and different $\theta$.}\label{Fig2}
\end{figure}
\begin{figure}[!htbp]
\centering
\includegraphics[width=0.9\columnwidth]{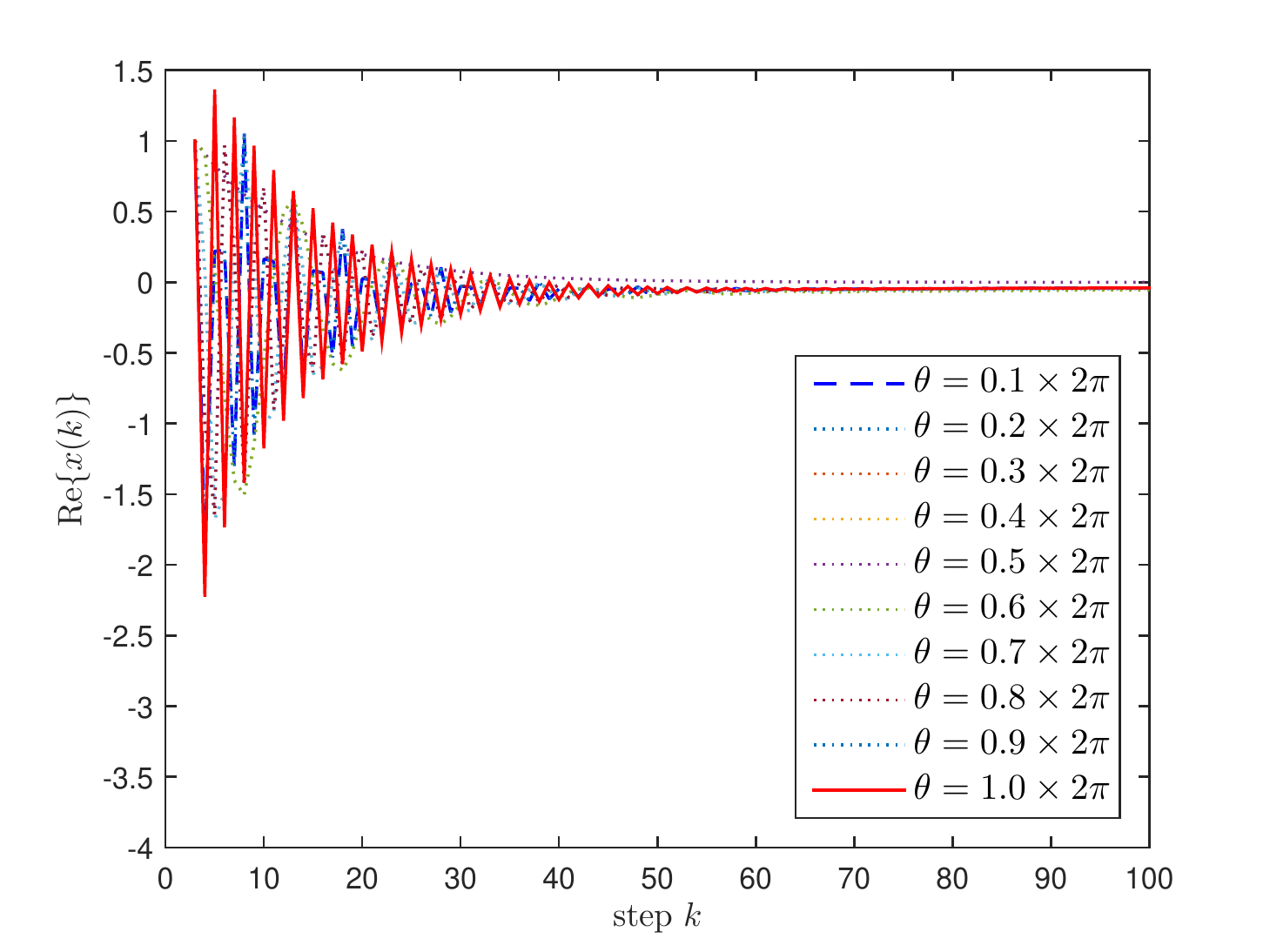}
\caption{The real part of $x(k)$ with $\rho=1.1$ and different $\theta$.}\label{Fig3}
\end{figure}

From the obtained figures, it can be observed that the illustrated results confirm our theoretical result firmly.
\end{example}

\section{CONCLUSIONS}\label{Section 5}
In this paper, some potential properties of the sampling free $N$-transform have been reviewed and explored. After providing the related definitions in order, 14 lemmas were summarized and 14 theorems were presented to exhibit its special characteristics. Additionally, some related problems have been investigated based on the elaborated properties, including calculating $N$-transform of some function and deriving the equivalent model of the nabla discrete fractional order system. It is expected that this work could arouse our interest on the nabla discrete fractional order system.

\section*{ACKNOWLEDGEMENTS}
The work described in this paper was supported by the National Natural Science Foundation of China (61601431, 61573332), the Anhui Provincial Natural Science Foundation (1708085QF141) and the funds of China Scholarship Council (201706340089, 201806345002). This draft was updated after it was accepted by FDTA 2019.

\bibliographystyle{asmems4}
\bibliography{database}

\end{document}